\documentclass[10pt]{article}
\usepackage{geometry}                % See geometry.pdf to learn the layout options. There are lots.
\usepackage{graphicx}
\usepackage{amssymb,amsthm}
\usepackage{epstopdf}
\usepackage{pdfsync}

\newtheorem{definition}{Definition}[section]
\newtheorem{lemma}[definition]{Lemma}
\newtheorem{theorem}[definition]{Theorem}
\newtheorem{proposition}[definition]{Proposition}

\newtheorem{remark}[definition]{Remark}
\newtheorem{example}[definition]{Example}

\def\Proof{\noindent{\it Proof.}\ }

\def\e{\varepsilon}
\def\dx{\,dx}

\def\ZZ{\mathbb{Z}}
\def\NN{\mathbb{N}}
\def\rr{\mathbb{R}}

\def \trait (#1) (#2) (#3){\vrule width #1pt height #2pt depth #3pt}
\def \qed{\hfill
        \trait (0.1) (6) (0)
        \trait (6) (0.1) (0)
        \kern-6pt
        \trait (6) (6) (-5.9)
        \trait (0.1) (6) (0)
\medskip}

\begin{document}
\title{Discrete double-porosity models}
\author{
{\sc Andrea Braides}
\\ Dipartimento di Matematica\\
 Universit\`a di Roma `Tor Vergata'\\
via della Ricerca Scientifica\\ 00133 Rome, Italy\\
\and  {\sc Valeria Chiad\`o Piat}
\\ Dipartimento di Matematica
\\ Politecnico di Torino \\  corso Duca degli Abruzzi 24\\
10129 Torino, Italy
\and
{\sc Andrey Piatnitski}
\\ Faculty of Technology, Narvik University College,
\\ HiN, Postbox 385, 8505 Narvik, Norway
\\ and P.N. Lebedev Physical Institute
\\ RAS, Leninski prospect 53, Moscow 119921, Russia}\date{}                                           % Activate to display a given date or no date

\maketitle

\section{Introduction}
Variational theories of double-porosity models can be derived by homogenization of high-contrast periodic media
(see \cite{BCP}).
Typically, we have one or more strong phases (i.e., uniformly elliptic energies on periodic connected domains)
and a weak phase with a small ellipticity constant, coupled with some lower-order term.
In the simplest case of quadratic energies, this amounts to considering energies of the form
\begin{equation}\label{unno}
\sum_{j=1}^N \int_{\Omega\cap \e C_j} |\nabla u|^2\dx +\e^2\int_{\Omega\cap\e C_0}|\nabla u|^2\d
+K\int_\Omega |u-u_0|^2\dx,
\end{equation}
where $\e$ is a geometric parameter representing the scale of the media. The strong components are modeled
for $j=1,\ldots, N$ by periodic connected Lipschitz sets $C_j$ of $\rr^d$ with pairwise disjoint closures; in this notation $C_0$ is their complement and represents the weak phase. Note that we may have $N>1$ only in dimension $d\ge 3$, while in dimension $d=2$ this model represents a single strong medium with weak inclusions (i.e., the set $C_0$ is composed of disjoint bounded components).  In dimension $d=1$ the energy trivializes since $C_0$ must be empty and the energy is then $\e$-independent. The scaling $\e^2$ in front of the weak phase is chosen so that the limit is non trivial; the analyses for all other scalings are derived from this one by comparison.

If we let $\e\to0$ these energies are approximated by their $\Gamma$-limit (\cite{GCB,Handbook}), which combines the homogenized energies of each strong medium (which exist by \cite{ACDP,BCP}) and a coupling term.
Note that the energies above are not strongly coercive in $L^2$. They are weakly coercive in $L^2$, but their limit is more meaningful if computed with respect to some topology which takes into account the strong limit of the functions on each strong component (or, more precisely, of the extensions of the restrictions of functions on each $\e C_j$, which are taken into account by the fundamental lemma by Acerbi, Chiad\`o Piat, Dal Maso and Percivale \cite{ACDP}).
In this way a convergence $u_\e\to (u_1,\ldots, u_N)$ is defined, the limit then depends on these $N$ independent functions, and takes the form
\begin{equation}\label{duue}
\sum_{j=1}^N \int_{\Omega}\Bigl( \langle A^j_{\rm hom}\nabla u_j, \nabla u_j\rangle+Kc_j |u_j-u_0|^2\Bigr)\dx
+\int_\Omega \varphi(u_0, u_1,\ldots, u_N)\dx,
\end{equation}
where $c_j$ are the volume fractions of the strong components and $\varphi$ is a quadratic function taking into account the interaction between the macroscopic phases. Note that the lower-order term is not continuous with respect to the convergence $u_\e\to(u_1,\ldots, u_N)$, which explains the appearance of an interaction term, whose computation in general involves a minimum problem on the weak phase $C_0$ (see \cite{BCP} for results in the general framework of $p$-growth Sobolev energies, \cite{Solci, Solci2} for perimeter energies, and \cite{BS-UMI} for free-discontinuity problems).

In this paper we derive double-porosity models from very simple atomistic interactions. Again in the case of quadratic
energies, we may write the microscopic energies (in the case of the cubic lattice $\ZZ^d$) as
\begin{equation}\label{trre}
\sum_{(\alpha,\beta)\in\e {\cal N}_1\cap(\Omega\times\Omega)}\e^d\Bigl|{u_\alpha-u_\beta\over\e}\Bigr|^2
+\e^2\sum_{(\alpha,\beta)\in\e {\cal N}_0\cap(\Omega\times\Omega)}\e^d\Bigl|{u_\alpha-u_\beta\over\e}\Bigr|^2
+K\sum_{\alpha\in \e\ZZ^d\cap\Omega}\e^d|u_\alpha- (u_0)_\alpha|^2.
\end{equation}
For explicatory purposes here we use a simplified notation with respect to the rest of the paper, and we denote by ${\cal N}_1$ the set of pairs in $\ZZ^d\times \ZZ^d$ between which we have strong interactions, and by
${\cal N}_0$ the set of pairs in $\ZZ^d\times \ZZ^d$ between which we have weak interactions. The energies depend on discrete functions whose values $u_\alpha$ are defined for $\alpha\in\e\ZZ^d$. Connected graphs of points linked by strong interactions play the role that in the continuum models is played by the sets $C_j$ ($j\neq 0$).
In order to define a limit continuous parameter, we have to suppose that at least one infinite such connected graph
exists, in which case
we may take the limit of (extension of) piecewise-constant interpolations of $u_\alpha$ on this graph as a continuous parameter. If we have more such infinite connected graphs the limit is described again by an array $(u_1,\ldots, u_N)$. In the more precise notation of this paper below we directly define the (analogs of the) $C_j$ and derive the corresponding strong and weak interaction accordingly. Note that weak interactions in ${\cal N}_0$ are due either to the existence of ``weak sites'' or to weak bonds between different ``strong components'', and, if we have more than one strong graph, the interactions in ${\cal N}_0$ are present also in the absence of a weak component. Under such assumptions, the limit is again of the form (\ref{duue}). In the paper we treat the general case of vector-valued $u_\alpha$, where the energy densities are given by some asymptotic formulas.

From the description of $\Gamma$-limits we also derive a dynamic results using the theory of minimizing movements. Under convexity assumptions, in that framework, the behaviour of gradient flows of a sequence $F_\e$
is described by the analysis of discrete trajectories $u^{\tau,\e}_j$ defined iteratively as minimizers of
$$
F_\e(u)+{1\over 2\tau}\|u-u^{\tau,\e}_{j-1}\|^2,
$$
with $\tau$ a time step (in our case the norm is the $L^2$-norm for discrete functions).
In our case, we take as $F_\e$ the energies above without lower-order term (i.e., with $K=0$).
We first show the strong convergence of $u^{\tau,\e}_{j}$ as $\e\to 0$. In this way, we can treat
these functions as fixed and apply the static limit results with $K={1/(2\tau)}$.
We may then follow the theory for equi-coercive and convex functionals,
for which gradient-flow dynamics commutes with the static limit (see \cite{2013LN,AG,AGS}).
As a result we show that the limit is described by a coupled system of PDEs (in the strong phases)
and ODEs (parameterized by the weak phase). It is interesting to note that this latter parameterization
is easily obtained by a discrete two-scale limit of the trajectories.

We finally note that in the discrete environment the topological requirements governing the interactions between the strong and weak phases are substituted by assumptions on long-range interactions. In particular, for discrete systems with second-neighbour interactions we may have a multi-phase limit also in dimension one.

\section{Notation}\label{Not}
The numbers $d$, $m$, $T$ and $N$ are positive integers.
We introduce a $T$ periodic {\em label function} $J:\ZZ^d\to\{0,1\ldots, N\}$, and the corresponding sets of sites
$$
A_j=\{ k\in\ZZ^d: J(k)=j\},\qquad j=0,\ldots, N.
$$

Sites interact through possibly long (but finite)-range interactions, whose range is defined through finite subsets
$P_j\subset \ZZ^d$, $j=0,\ldots, N$. Each $P_j$ is symmetric and $0\in P_j$.

We say that two points $k,k'\in A_j$ are $P_j$-{\em connected} in $A_j$ if there exists a path $\{k_n\}_{n=0,\ldots,K}$
such that $k_n\in A_j$, $k_0=k$, $k_K=k'$ and $k_n-k_{n-1}\in P_j$.

We suppose that there exists a unique infinite $P_j$-connected component of each $A_j$ for $j=1,\ldots, N$, which we denote by $C_j$. Note that we do not make any such assumption for $A_0$.

We consider the following sets of bonds between sites in $\ZZ^d$:
for $j=1,\ldots, N$
$$
N_j=\{(k,k'): k,k'\in A_j,  %C_j,
k-k'\in P_j\setminus \{0\}\};
$$
for $j=0$
$$
N_0=\{(k,k'): k-k'\in P_0\setminus\{0\}, J(k)J(k')=0\hbox{ or } J(k)\neq J(k')\}.
$$
Note that the set $N_0$ takes into account of interactions not only among points of the set $A_0$, but also among pair of points in different $A_j$. A more refined model could be introduced by defining range of interactions $P_{ij}$ and the corresponding sets $N_{ij}$, in which case the sets $N_j$ would correspond to $N_{jj}$ for $j=1,\ldots, N$ and $N_0$ the union of the remaining sets. However, for simplicity of presentation we limit our notation to a single index.

We consider interaction energy densities $f:\ZZ^d\times\ZZ^d\times\rr^m\to \rr$ and $g:\ZZ^d\times\rr^m\to \rr$.
Note that the values of the function $f(k,k',z)$ will be considered only for  $(k,k')$ belonging to some $N_j$.
The  functions $f$ and $g$ satisfy the following conditions:
$f(k,k',z)= f(k',k,z)$ (this is not a restriction up to substituting $f(k,k',z)$ with ${1\over 2}(f(k,k',z)+f(k',k,z)$)
and there exists $p>1$ such that
\begin{equation}\label{c1}
c(|z|^p-1)\le f(k,k',z)\le C(|z|^p+1)\qquad 0\le f(k,k',z),
\end{equation}
\begin{equation}\label{c2}
|f(k,k',z)-f(k,k',z')|\le C|z-z'|\bigl(1+|z|^{p-1}+|z'|^{p-1}\bigr)
\end{equation}
\begin{equation}\label{c3}
f(k,k',\cdot) \hbox{ is positively homogeneous of degree $p$ if }(k,k')\in N_0
\end{equation}
\begin{equation}\label{c4}
0\le g(k,u)\le C(|z|^p+1)
\end{equation}
\begin{equation}\label{c5}
|g(k,z)-g(k,z')|\le C|z-z'|\bigl(1+|z|^{p-1}+|z'|^{p-1}\bigr).
\end{equation}

Given $\Omega$ a bounded regular open subset of $\rr^d$, we define the energies
\begin{eqnarray}\nonumber\label{fep}
F_\e(u)=F_\e\Bigl(u,{1\over\e}\Omega\Bigr)&=&\sum_{j=1}^N\sum_{(k,k')\in {\mathcal N}^\e_j(\Omega)} \e^d f\Bigl(k,k', {u_k-u_{k'}\over\e}\Bigr)\\
&&+
\sum_{(k,k')\in {\mathcal N}^\e_0(\Omega)} \  \e^{d+p} f\Bigl(k,k', {u_k-u_{k'}\over\e}\Bigr)+\sum_{k\in Z^\e(\Omega)} \e^d g(k, u_k),
\end{eqnarray}
where
%\begin{equation}
%N^\e_j(\Omega)= \Bigl\{ (k,k')\in (A_j\times A_j)\cap {1\over \e} (\Omega\times\Omega): k-k'\in P_j, k\neq k'\Bigr\},
%\end{equation}
\begin{equation}
{\mathcal N}^\e_j(\Omega)=N_j\cap {1\over \e} (\Omega\times\Omega), j=0,\ldots,N,
\qquad\qquad
Z^\e(\Omega)=\ZZ^d\cap {1\over \e} \Omega.
\end{equation}
The energy is defined on discrete functions $u: {1\over\e}\Omega\cap\ZZ^d\to\rr^m$.

The first sum in the energy takes into account all interactions between points in $A_j$ ({\em hard phases}), which are supposed to scale differently than those between points in $A_0$ ({\em soft phase}) or in different phases. The latter are contained in the second sum. The third sum is a zero-order term taking into account with the same scaling all types of phases.

Note that the first sum may take into account also points in $A_j\setminus C_j$, which form ``islands'' of the hard phase $P_j$-disconnected from the corresponding infinite component. Furthermore, in this energy we may have sites that do not interact at all with hard phases.

\section{Homogenization of ``perforated'' discrete domains}
In this section we separately consider the interactions in each infinite connected component of hard phase introduced above. To that end we fix one of the indices $j$, $j>0$, dropping it in the notation of this section (in particular we use the symbol $C$ in place of  $C_j$, etc.), and define the energies
\begin{eqnarray}\label{cale}
{\cal F}_\e(u)={\cal F}_\e\Bigl(u, {1\over\e}\Omega\Bigr)=\sum_{(k,k')\in N^\e_C(\Omega)} \e^d f\Bigl(k,k', {u_k-u_{k'}\over\e}\Bigr)\,,
\end{eqnarray}
where
\begin{equation}
N^\e_C(\Omega)= \Bigl\{ (k,k')\in (C\times C)\cap {1\over \e} (\Omega\times\Omega): k-k'\in P, k\neq k'\Bigr\},
\end{equation}
We also introduce the notation $C^\e(\Omega)= C\cap{1\over\e}\Omega$.

\begin{definition}\label{pci}\rm
The {\em piecewise-constant interpolation} of a function $u:\ZZ^d\cap{1\over\e}\Omega\to\rr^m$, $k\mapsto u_k$ is defined as
$$
u(x)= u_{\lfloor x/\e\rfloor},
$$
where $\lfloor y\rfloor= (\lfloor y_1\rfloor,\ldots, \lfloor y_d\rfloor)$ and $\lfloor s\rfloor$ stands for the integer part of $s$.
The {\em convergence} of a sequence $(u^\e)$ of discrete functions is understood as the $L^1_{\rm loc}(\Omega)$ convergence of these piecewise-constant interpolations. Note that, since we consider local convergence in $\Omega$, the value of $u(x)$ close to the boundary in not involved in the convergence process.
\end{definition}

We prove an extension and compactness lemma with respect to the convergence of piecewise-constant interpolations.

\begin{lemma}[extension and compactness]\label{exc}
Let $u^\e: {1\over\e}\Omega\to\rr^m$ be a sequence such that
\begin{equation}
\sup_\e\Bigl\{ \sum_{(k,k')\in N^\e_C(\Omega)} \e^d\Bigl|{u^\e_k-u^\e_{k'}\over\e}\Bigr|^p+\sum_{k\in C^\e(\Omega)}
\e^d|u^\e_k|\Bigr\}<+\infty.
\end{equation}
Then there exists a sequence $\widetilde u^\e: {1\over\e}\Omega\to\rr^m$ such that $\widetilde u^\e_k= u^\e_k$ if
$k\in C^\e(\Omega)$ and {\rm dist}$(k, \partial{1\over\e}\Omega)> C(T,p,d,m)$, with $u^\e$ converging to $u\in W^{1,p}(\Omega)$ up to subsequences.
\end{lemma}

\Proof It suffices to treat the scalar case $m=1$, up to arguing component-wise.

With fixed $i\in \ZZ^d$ we consider a periodicity cell $Y_i= iT+Y$, where $Y=[0,T)^d\cap \ZZ^d$.
If we consider $k\in C\cap Y_i$ and $k'\in C\cap Y'_i$,
where $Y'_i$ is either $Y_i$ or a neighboring periodicity cell, then the minimal path in $C$ connecting $k$ and $k'$ lies in a periodicity cube $\widetilde Y_i= iT+ [-DT, (D+1)T)^d$ for some positive integer $D$.
We suppose that such  $\widetilde Y_i$ is contained in
${1\over \e}\Omega$. This holds if
\begin{equation}\label{disto}
\hbox{ {\rm dist}}\Bigl(Y_i, \partial{1\over\e}\Omega\Bigr)> C(T)
\end{equation}
for some $C(T)$.

We define
$$
\widetilde u^\e_k={1\over \#(C\cap Y)} \sum_{l\in C\cap Y_i} u^\e_l\qquad \hbox{ for }k\in Y_i\setminus C.
$$
For $k\in Y_i$ and $|k-k'|=1$ (in the notation above $k'\in Y'_i$) we have
\begin{eqnarray*}
\e^d\Bigl|{\widetilde u^\e_k-\widetilde u^\e_{k'}\over\e}\Bigr|^p
\le \e^{d-p}\Bigl|\max_{Y_i\cup Y_i'}u^\e- \min_{Y_i\cup Y_i'}u^\e\Bigr|^p
= \e^{d-p}|u^\e_l- u^\e_{l'}|^p
\end{eqnarray*}
for some $l,l'\in Y_i\cup Y_i'$. We then may take a path $\{u_{l_n}\}_{n=1\ldots,N}$ in $C$ connecting $l$ and $l'$ lying in $ \widetilde Y_i$. We then have
\begin{eqnarray*}
\e^d\Bigl|{\widetilde u^\e_k-\widetilde u^\e_{k'}\over\e}\Bigr|^p
\le C\sum_{n=1}^N\e^{d}\Bigl|{u^\e_{l_n}- u^\e_{l_{n-1}}\over\e}\Bigr|^p\le
 C\sum_{j-j'\in P, j,j'\in \widetilde Y_i\cap C}\e^{d}\Bigl|{u^\e_j- u^\e_{j'}\over\e}\Bigr|^p.
\end{eqnarray*}
Summing up in $k,k'$ we obtain
\begin{eqnarray*}
\sum_{|k-k'|=1, k\in Y_i}\e^d\Bigl|{\widetilde u^\e_k-\widetilde u^\e_{k'}\over\e}\Bigr|^p
\le C T^d \sum_{j-j'\in P, j,j'\in \widetilde Y_i\cap C}\e^{d}\Bigl|{u^\e_j- u^\e_{j'}\over\e}\Bigr|^p.
\end{eqnarray*}
and
\begin{eqnarray}\label{wupe}
\sum_{|k-k'|=1, k,k'\in {1\over\e}\widetilde\Omega_\e}\e^d\Bigl|{\widetilde u^\e_k-\widetilde u^\e_{k'}\over\e}\Bigr|^p
\le C D^dT^d \sum_{(j,j')\in N^\e_C(\Omega)}\e^{d}\Bigl|{u^\e_j- u^\e_{j'}\over\e}\Bigr|^p,
\end{eqnarray}
where
$$
\widetilde\Omega_\e= \bigcup\Bigl\{ \e Y_i: \hbox{(\ref{disto}) holds} \Bigr\}.
$$

Trivially, we also have the estimate
$$
\sum_{k\in Y_i}|\widetilde u^\e_k|\le
\sum_{k\in Y_i\cap C}| u^\e_k|+ \sum_{k\in Y_i\setminus C}|\widetilde u^\e_k|
={T^d\over \#(C\cap Y)}\sum_{k\in Y_i\cap C}| u^\e_k|.
$$
These two estimates ensure the pre-compactness of $\widetilde u^\e$ in $L^1(\Omega')$ for all $\Omega'\subset\!\subset\Omega$ and that every its cluster point is in $W^{1,p}(\Omega)$ by the uniformity of the estimates (\ref{wupe}) (see \cite{AC}).
\qed

\begin{theorem}[homogenization on discrete perforated domains]\label{hdpd}
The energies ${\cal F}_\e$ defined in {\rm(\ref{cale})} $\Gamma$-converge with respect to the $L^1_{\rm loc}(\Omega;\rr^m)$ topology to the energy
\begin{equation}
{\cal F}_{\hom}(u)=\int_\Omega f_{\hom}(\nabla u)\dx,
\end{equation}
defined on $W^{1,p}(\Omega;\rr^m)$, where the energy density $f_{\hom}$ satisfies
\begin{equation}\label{foro}
f_{\hom}(\xi)=\lim_{K\to+\infty} \inf\Bigl\{ {\cal F}(\xi x+v,(0,K)^d): v_k=0
\hbox{ in a neighborhood of } \partial (0,K)^d\Bigr\}.
\end{equation}
\end{theorem}

\Proof The proof follows the one in the case $C=\ZZ^d$ contained in \cite{AC}, and therefore we have the coerciveness condition $f(k,k',z)\ge C(|z|^p-1)$ whenever $|k-k'|=1$. That condition is used only to obtain pre-compactness of sequences with bounded energy, and is substituted by the previous lemma.

The proof can also be obtained by directly using the homogenization result of \cite{AC} applied to
${\cal F}^\eta_\e= {\cal F}_\e+\eta\, G$, where
$$
G(u)=\sum_{|k-k'|=1, k,k'\in{1\over\e}\Omega} \e^d\Bigl|{ u_k-u_{k'}\over\e}\Bigr|^p,
$$
obtaining limit energies $${\cal F}^\eta_{\hom}(u)=\int_\Omega f^\eta_{\hom}(\nabla u)\dx.$$ By comparison we obtain the existence of the desired
$\Gamma$-limit and the equality
$$
{\cal F}_{\hom}(u)=\inf_{\eta>0} {\cal F}^\eta_{\hom}(u)=\int_\Omega \inf_{\eta>0}
 f^\eta_{\hom}(\nabla u)\dx.
$$
Once this integral representation is shown to hold, standard arguments allow to conclude the validity of formula (\ref{foro}) (see \cite{BDF}).
\qed

\section{Definition of the interaction term}
The homogenization result in Theorem \ref{hdpd} will describe the contribution of the hard phases to the limiting behavior of energies $F_\e$. We now characterize their interactions with the soft phase.

For all $M$ positive integer and $z_1,\ldots, z_N\in \rr^m$  we define the minimum problem
\begin{equation}
\varphi_M(z_1,\ldots, z_N) ={1\over M^d}
\inf\Bigl\{ \sum_{(k,k')\in N_0(Q_M)} f(k,k', {v_k-v_{k'}})+\sum_{k\in Z(Q_M)} g(k, v_k): v\in{\cal V}_M\Bigr\},
\end{equation}
where
\begin{equation}
Q_M=\Bigl[-{M\over 2},{M\over 2}\Bigr)^d, \qquad N_0(Q_M)= N_0\cap (Q_M\times Q_M),\qquad Z(Q_M)= \ZZ^d\cap Q_M,
\end{equation}
and the infimum is taken over the set ${\cal V}_M= {\cal V}_M(z_1,\ldots, z_N)$ of all $v$ that are constant on each connected component of $A_j\cap Q_M$ and $v=z_j$
on $C_j$ for $j=1,\ldots N$.

\begin{proposition}\label{vf}
There exists the limit $\varphi$ of $\varphi_M$ uniformly on compact subsets of $\rr^{mN}$.
\end{proposition}

\Proof Note preliminarily that by the positive homogeneity condition for $f$ we have
$$
|f(k,k',z)- f(k,k', z')|\le C|z-z'|(|z|^{p-1}+|z'|^{p-1})
$$
for $(k,k')\in N_0$. Let $v$ be a test function for $\varphi_M(z_1,\ldots, z_N)$.  In order to estimate $\varphi_M(z'_1,z_2,\ldots, z_N)$ we use as test function
$$
v'_k=\cases{ z'_1 & if $k\in C_1$\cr
v_k & otherwise.}
$$
We then have
\begin{eqnarray*}
&&\biggl|\sum_{(k,k')\in N_0(Q_M)} f(k,k', {v'_k-v'_{k'}})+\sum_{k\in Z(Q_M)} g(k, v'_k)\\
&&\qquad\qquad-
\sum_{(k,k')\in N_0(Q_M)} f(k,k', {v_k-v_{k'}})-\sum_{k\in Z(Q_M)} g(k, v_k)\biggr|
\\
&\le& 2\sum_{(k,k')\in N_0(Q_M), k\in C_1}\Bigl| f(k,k', {v'_k-v'_{k'}})-
f(k,k', {v_k-v_{k'}})\Bigr|+\
\sum_{k\in C_1\cap Q_M}|g(k, z_1)-g(k, z_1')\Bigr|
\\
&\le& 2\sum_{(k,k')\in N_0(Q_M), k\in C_1}\Bigl| f(k,k', {z'_1-v_{k'}})-
f(k,k', {z_1-v_{k'}})\Bigr|+\
\sum_{k\in C_1\cap Q_M}|g(k, z_1)-g(k, z_1')\Bigr|
\end{eqnarray*}
By (\ref{c5}) the second sum can be simply estimated by $CM^d|z_1-z_1'|\Bigl(1+|z_1|^{p-1}+|z_1'|^{p-1}\Bigr)$.
As for the first sum, we have
\begin{eqnarray*}
&&\sum_{(k,k')\in N_0(Q_M), k\in C_1}\Bigl| f(k,k', {z'_1-v_{k'}})-
f(k,k', {z_1-v_{k'}})\Bigr|\\
&\le& \sum_{(k,k')\in N_0(Q_M), k\in C_1}
C|z_1-z_1'|\Bigl(|z'_1-z_1|^{p-1}+|v_k-v_{k'}|^{p-1}\Bigr)\\
&\le&
CM^d|z_1-z_1'|^p +C|z_1-z_1'|\sum_{(k,k')\in N_0(Q_M), k\in C_1}|v_k-v_{k'}|^{p-1}\\
&\le&
CM^d|z_1-z_1'|^p +C|z_1-z_1'|M^{d\over p}\Bigl(\sum_{(k,k')\in N_0(Q_M), k\in C_1}|v_k-v_{k'}|^{p}\Bigr)^{{p-1\over p}}
\\
&\le&
CM^d|z_1-z_1'|^p +C|z_1-z_1'|M^{d\over p}\Bigl(\sum_{(k,k')\in N_0(Q_M), k\in C_1}f(k,k',v_k-v_{k'})\Bigr)^{{p-1\over p}}.
\end{eqnarray*}
By the arbitrariness of $v$, taking infima we conclude that
\begin{eqnarray*}
&&|\varphi_M(z'_1,\ldots, z_N)-\varphi_M(z_1,\ldots, z_N)|\\
&&\le C|z_1-z_1'|\Bigl(|z_1-z_1'|^{p-1}+ (\varphi_M(z_1,\ldots, z_N) )^{{p-1\over p}}+\Bigl(1+|z_1|^{p-1}+|z_1'|^{p-1}\Bigr)\Bigr).
\end{eqnarray*}
Furthermore, by taking as test function $v=0$ on the complement of the $\bigcup_jC_j$ we have the estimate
$$
\varphi_M(z_1,\ldots, z_N)\le C\Bigr(1+\sum_j|z_j|^p\Bigl).
$$
These estimates give equiboundedness and equicontinuity of the family $\varphi_M$ on bounded subsets.
By Ascoli-Arzel\`a's theorem, to conclude it suffices to show that the whole sequence $\varphi_M$ converges point wise. To this end, we note that for integer $K$ and $M$ we have\nobreak

(i) $\varphi_{KM}\ge \varphi_M$;

(ii) $M^d\varphi_M\le K^d\le \varphi_K$ if $M\le K$.

\goodbreak
By (i), with fixed $M$ the sequence $\varphi_{M2^k}$ is increasing, and in particular
\begin{equation}\label{inco}
\varphi_{M2^k}\ge \varphi_M
\end{equation}
for all $k$.

Let $k$ be fixed; for all $K$ let $L_K=\lfloor K/M2^k\rfloor$, so that
$$
0\le K-L_KM2^k\le M2^k.
$$
Then, by (ii)
$$
({L_KM2^k})^d\varphi_{L_KM2^k}\le K^d \varphi_K,
$$
and by (\ref{inco})
$$
\varphi_K\ge \Bigl({L_KM2^k\over K}\Bigr)^d\varphi_{L_KM2^k}\ge  \Bigl({L_KM2^k\over K}\Bigr)^d\varphi_{M2^k}\ge  \Bigl({L_KM2^k\over K}\Bigr)^d\varphi_{M}.
$$
By taking first the liminf in $K$ and then the limsup in $M$ we obtain
$$
\liminf_K \varphi_K\ge \limsup_M \varphi_{M};
$$
that is, the thesis.
\qed

\begin{remark}\label{frame}\rm
Let $u^M\in {\cal V}_M$ be a sequence such that
$$
\lim_M{1\over M^d}\Bigl(\sum_{(k,k')\in N_0(Q_M)} f(k,k', {u^M_k-u^M_{k'}})+\sum_{k\in Z(Q_M)} g(k, u^M_k)\Bigr)=\varphi(z_1,\ldots, z_N) $$
then for every sequence of constants $R_M= o(M)$ we have
$$
\lim_M{1\over M^d}
\sum_{k,k'\in Q_M\setminus Q_{M-R_M}: k-k'\in P_0}  |u^M_k-u^M_{k'}|^p=0.
$$
Indeed, otherwise taking $u^M$ as test function for the problem defining $\varphi_{M-R_M}(z_1,\ldots, z_N)$,
we would obtain
$$
\limsup_M\varphi_{M-R_M}(z_1,\ldots, z_N)<\varphi(z_1,\ldots, z_N),
$$
which is a contradiction.
\end{remark}

We now prove that the function $\varphi$ introduced in Proposition \ref{vf} can be defined through minimum problems with additional boundary data. This will be useful in the computation of the upper bound for the $\Gamma$-limit.
We then define the boundary set of $Q_M$ as follows: we consider $R$ a fixed constant
%with $$R>\max\Bigl\{|p|: p\in\bigcup_{j=0}^N P_j\Bigr\},$$
such that for any two points $k$ and $k'\in Q_{M-R}$ connected in terms of $P_0$-interactions there exists a path
of $P_0$-interacting points contained in $Q_M$,
and $R$ larger than twice the diameter of each bounded connected component of any $A_j$ for $j=1,\ldots,N$.
We define $B_M$ as
\begin{eqnarray*}
B_M&=&\Bigl((Q_M\setminus Q_{M-R})\setminus \bigcup_{j=1}^NA_j\Bigr)\\
&&\cup\,\bigcup\{B: B\subset Q_M\setminus Q_{M-R} \hbox{ bounded connected component of } A_j\cap Q_M, j=1,\ldots,N\}.
\end{eqnarray*}
With this definition, we can set
\def\wif{\widetilde \varphi}
\begin{equation}\label{wif}
\wif_M(z_1,\ldots,z_N)={1\over M^d}
\inf\Bigl\{ \sum_{(k,k')\in N_0(Q_M)} f(k,k', {v_k-v_{k'}})+\sum_{k\in Z(Q_M)} g(k, v_k): v\in {\cal V}_M, v=0\hbox{ on } B_M\Bigr\}.
\end{equation}

\begin{proposition}
There exists the limit
$$
\lim_M \wif_M(z_1,\ldots,z_N)=\varphi(z_1,\ldots,z_N),
$$
uniformly on bounded subsets of $\rr^{mN}$, where $\varphi$ is defined in Proposition {\rm\ref{vf}}.
\end{proposition}

\Proof By the same argument as in Proposition {\rm\ref{vf}} we may show that the sequence is equibounded and equicontinuous on bounded sets. It is then sufficient to show the existence of the pointwise limit, and that this coincides with that of $\varphi_M$. To this end we will estimate $\wif_M$ in terms of $\varphi_M$.

Note that we may write $\varphi_M(z_1,\ldots,z_N)$ as the sum of two independent minimum problems, the first one
where only $k$ and $k'$ connected with $\bigcup_{j=1}^N C_j$ in $Q_M$ are taken into account, and the second one where the summation is done over all other indices (disconnected with $\bigcup_{j=1}^N C_j$). Note that the first one is actually a minimum, of which we choose a minimizer $v^M$, while the second one may be only an infimum. The latter infimum can be further decomposed into a sum of disjoint infimum problems over bounded connected component, the ones intersecting $Q_{M-R}$ being $T\ZZ^d$-translations of a finite family $\{I_l\}$ of subsets of $\ZZ^d$ by our choice of $R$; i.e.,
their value is
\begin{equation}\label{il}
\inf\Bigl\{ \sum_{(k,k')\in N_0(I_l)} f(k,k', {v_k-v_{k'}})+\sum_{k\in I_l} g(k, v_k): v:I_l\to\rr^m
\Bigr\},
\end{equation}
where the infimum is taken on those $v$ that are constant on each component of $A_j\cap I_l$ for $j=1,\ldots, N$.
This value is independent of $M$ and $z_1,\ldots, z_N$.
We denote by $w^l$ a ${1\over M}$-almost minimizer of problem (\ref{il}).

We define $\widetilde v^M\in {\cal V}_M$ with $\widetilde v^M=0$ in $B_M$ by setting
$$
\widetilde v^M_k=\cases{ 0 & if $k\in B_M$\cr
w^l_{k-K} & if $k\in K+I_l$ for some $K\in T\ZZ^d$ and $K+I_l\cap Q_{M-R}\neq \emptyset$\cr
v^M_k & otherwise.}
$$
Using $\widetilde v^M$ as a test function we can estimate, recalling (\ref{c3}),  (\ref{c1}) and  (\ref{c4}),
\begin{eqnarray*}
\wif_M(z_1,\ldots,z_N)&\le&\varphi_M(z_1,\ldots,z_N)\\
&&+{C\over M^d}\Bigl( \sum_{k {\scriptstyle\,\rm or\,} k'\in B_M, k-k'\in P_0}|\widetilde v^M_k-\widetilde v^M_{k'}|^p+
\sum_{k\in B_M} g(k,0)+M^{d-1}+\#(A_j\cap B_M)\Bigr)\\
&\le&\varphi_M(z_1,\ldots,z_N)+{C\over M^d}\Bigl( \sum_{k\not\in B_M: \exists k'\in B_M, k-k'\in P_0}|\widetilde v^M_k|^p+\#(B_M)\Bigr).
\end{eqnarray*}
By Poincar\`e inequality the sum can be estimated as
$$
\sum_{k\not\in B_M: \exists k'\in B_M, k-k'\in P_0}|\widetilde v^M_k|^p\le
C\Bigl(\#(B_M)+\sum_{k,k'\in Q_M\setminus Q_{M-2R}: k-k'\in P_0} |v^M_k-v^M_{k'}|^p\Bigl).
$$
Since this last sum tends to $0$ as $M\to+\infty$ by Remark \ref{frame}, we obtain
$$
\wif_M(z_1,\ldots,z_N)\le\varphi_M(z_1,\ldots,z_N)+o(1).
$$
Since the opposite inequality $\wif_M(z_1,\ldots,z_N)\ge\varphi_M(z_1,\ldots,z_N)$ trivially holds, we get
that $$\lim_M(\wif_M(z_1,\ldots,z_N)-\varphi_M(z_1,\ldots,z_N))=0$$
as desired.
\qed

\section{Statement of the convergence result}
We now have all the ingredients to characterize the asymptotic behavior of $F_\e$.

Thanks to the compactness Lemma \ref{exc}, we may define the convergence
\begin{equation}\label{conva}
u^\e\to (u_1,\ldots, u_N)
\end{equation}
as the $L^1_{\rm loc} (\Omega;\rr^m)$ convergence $\widetilde u^\e_j\to u_j$ of the extensions
of the restrictions of $u^\e$ to $C_j$, which is a compact convergence as ensured by that lemma.

The total contribution of the hard phases will be given separately by the contribution on the infinite connected components and the finite ones.
The first one is obtained by
computing independently the limit relative to each component
\begin{eqnarray}\label{calej}
F^j_\e(u)=\sum_{(k,k')\in N^\e_j(\Omega)} \e^d f\Bigl(k,k', {u_k-u_{k'}\over\e}\Bigr)\,,
\end{eqnarray}
where
\begin{equation}
N^\e_j(\Omega)= \Bigl\{ (k,k')\in (C_j\times C_j)\cap {1\over \e} (\Omega\times\Omega): k-k'\in P_j, k\neq k'\Bigr\},
\end{equation}
which is characterized by Theorem \ref{hdpd} as
\begin{equation}\label{fjh}
F^j_{\hom} (u)=\int_\Omega f^j_{\hom} (\nabla u)\dx.
\end{equation}

In order to characterize the contribution of the finite connected components of $A_j$, we can write
\begin{equation}
A_j\setminus C_j=\bigcup_{l\in I_j}(A^j_l+ T\ZZ^d),
\end{equation}
where, due to the periodicity of the media, $l$ runs over a finite set of indices
$I_j$, and $A^j_l+ T\ZZ^d$ and $A^j_{l'}+ T\ZZ^d$ are $P_j$-disconnected if $l\neq l'$.
To each such $A^j_l$ we associate the minimum value
\begin{equation}\label{mji}
m^j_l=\min\Bigl\{\sum_{k,k'\in A^j_l, k-k'\in P_j} f(k,k', {z_k-z_{k'}}): z:A^j_l\to \rr^m\Bigr\}.
\end{equation}
Note that we have no boundary conditions for the test functions $z$. The total contribution of the disconnected components will simply give the additive constant $m|\Omega|$, where
\begin{equation}\label{m}
m={1\over T^d} \sum_{j=1}^N\sum_{l\in I_j} m^j_l.
\end{equation}

In the previous section we have introduced the energy density $\varphi$, which describes the interactions between the hard phases. Taking all contribution into account, we may state the following convergence result.

\begin{theorem}[double-porosity homogenization]\label{dph}
Let $\Omega$ be a Lipschitz bounded open set, and let $F_\e$ be defined by {\rm(\ref{fep})} with the notation of Section {\rm\ref{Not}}. Then there exists the $\Gamma$-limit of $F_\e$ with respect to the convergence {\rm(\ref{conva})}
and it equals
\begin{equation}\label{maineq}
F_{\hom}(u_1,\ldots, u_N)=\sum_{j=1}^N\int_\Omega f^j_{\hom} (\nabla u_j)\dx+m|\Omega|+
\int_\Omega \varphi(u_1,\ldots, u_N)\dx
\end{equation}
on functions $u=(u_1,\ldots, u_N)\in (W^{1,p}(\Omega;\rr^m))^N$,
where $\varphi$ is defined in Proposition {\rm \ref{vf}}, $f^j_{\hom}$ are defined by  {\rm(\ref{fjh})}, and
$m$ is given by {\rm(\ref{m})}.
\end{theorem}

The proof of this result will be subdivided into a lower and an upper bound in the next sections.

\begin{remark}[non-homogeneous lower-order term]\rm
In our hypotheses the lower-order term $g$ depends on the fast variable $k$, which is integrated out in the limit.
We may easily include a measurable dependence on the slow variable $\e k$, by assuming $g= g(x,k,z)$ a Carath\'eodory function (this covers in particular the case $g=g(x,z)$ and substitute the last sum in (\ref{fep}) by
$$
\sum_{k\in Z^\e(\Omega)} \e^d g(\e k,k, u_k).
$$
Correspondingly, in Theorem \ref{dph} the integrand in the last term in (\ref{maineq}) must be substituted by
$\varphi(x,u_1,\ldots, u_N)$, where the definition of this last function is the same but taking $g(x,k,z)$ in place of $g(k,z)$, so that $x$ simply acts as a parameter.
\end{remark}

\begin{example}[simple one-dimensional energies]\label{exh}\rm
We give two examples of one-dimensional energies with a non-trivial double-porosity limit due to next-to-nearest neighbour interactions.

(1) We consider $d=1$, $\Omega=(0,1)$ and the energies
$$
\sum_{i=1}^{\lfloor1/\e\rfloor-1} \e\Bigl|{u_{i+1}-u_{i-1}\over \e}\Bigr|^2+ \e^2
\sum_{i=1}^{\lfloor1/\e\rfloor} \e\Bigl|{u_{i}-u_{i-1}\over \e}\Bigr|^2.
$$
In this case $C_1$ and $C_2$ are the sets of even and odd integers, and $C_0=\emptyset$.
We have $g=0$ and the definition of $\varphi$ is trivial; the limit is
$$
F_{\hom}(u_1,u_2)=2\int_{(0,1)}|u_1'|^2\,dx+ 2\int_{(0,1)}|u_2'|^2\,dx+ \int_{(0,1)}|u_1-u_2|^2\,dx
$$
(note the abuse of notation for $u_i$). Note that the second sum of the discrete energy can be interpreted as the $L^2$-norm of the difference between even and odd interpolations of $u$

(2) We consider $d=1$, $\Omega=(0,1)$ and the energies
$$
\sum_{i=0}^{{1\over 2}\lfloor1/\e\rfloor-2} \e\Bigl|{u_{2i+2}-u_{2i}\over \e}\Bigr|^2+ \e^2
\sum_{i=1}^{\lfloor1/\e\rfloor} \e\Bigl|{u_{i}-u_{i-1}\over \e}\Bigr|^2 +
\sum_{i=1}^{\lfloor1/\e\rfloor} \e |u_i- u^0_i|^2.
$$
In this case, $C_1$ is the set of even integers, $C_0$ is the set of odd integers, and we may take
$g(x,z)= |z-u_0(x)|^2$ (we take $u_0$ a fixed $L^2$-function and $\{u^0_i\}$
an interpolation strongly converging to $u_0$). Correspondingly,
$$
\varphi(x,u)= {1\over 3} |u-u_0(x)|^2,
$$
and the limit is
$$
F_{\hom}(u)=2\int_{(0,1)}|u'|^2\,dx+ {1\over 3}\int_{(0,1)}|u-u_0(x)|^2\,dx
$$
(in this case we only have one parameter in the continuum).
\end{example}

\section{Lower bound}
Let $u^\e$ be such that $\sup_\e F_\e(u^\e)<+\infty$ and $u^\e\to u=(u_1,\ldots, u_N)$ with respect to convergence {\rm(\ref{conva})}.

We may then rewrite
\begin{eqnarray}\nonumber
F_\e(u^\e)&\ge&\sum_{j=1}^N F^j_\e(u^\e) +\sum_{j=1}^N\sum_{A^j_l\subset {1\over\e}\Omega}
\sum_{\begin{array}{cc}\scriptstyle k,k'\in A^j_l\\ [-1mm]\scriptstyle
k-k'\in P_j\end{array}}\e^d f\Bigl(k,k', {u^\e_k-u^\e_{k'}\over\e}\Bigr)\\
&&+\sum_{Q^i_M\subset {1\over\e}\Omega}\e^d
\Biggl(\sum_{(k,k')\in N_0(Q^i_M)}\e^p f\Bigl(k,k',{u^\e_k-u^\e_{k'}\over\e}\Bigr)+\sum_{k\in Z(Q^i_M)} g(k, u^\e_k)\Biggr),\label{deco}
\end{eqnarray}
where
$$
Q^i_M=Q_M+Mi, \qquad N_0(Q^i_M)= N_0\cap (Q^i_M\times Q^i_M),\qquad Z(Q^i_M)= \ZZ^d\cap Q^i_M,
$$
for $i\in\ZZ^d$.

The second term in (\ref{deco}) is estimated by taking the minimum over all $z_k$ in the place of $u^\e_k/\e$,
obtaining
\begin{eqnarray}\nonumber
\sum_{j=1}^N\sum_{A^j_l\subset {1\over\e}\Omega}
\sum_{\begin{array}{cc}\scriptstyle k,k'\in A^j_l\\ [-1mm]\scriptstyle
k-k'\in P_j\end{array}}\e^d f\Bigl(k,k', {u^\e_k-u^\e_{k'}\over\e}\Bigr)&\ge&
\e^d\sum_{j=1}^N\sum_{A^j_l\subset {1\over\e}\Omega}  m^j_l\\
=\ \e^d\sum_{j=1}^N {|\Omega|\over\e^d T^d}\sum_{l\in I_J} m^j_l+o(1)&=& m|\Omega|+o(1).
\label{sete}
\end{eqnarray}

In order to estimate the last term in (\ref{deco}) we estimate separately
$$
\sum_{(k,k')\in N_0(Q^i_M)}\e^p f\Bigl(k,k',{u^\e_k-u^\e_{k'}\over\e}\Bigr)+\sum_{k\in Z(Q^i_M)} g(k, u^\e_k)
$$
for each fixed $i$. To this end, we consider the function $u^{\e,i}$ defined by
$$
u^{\e,i}_k={1\over \#(C_j\cap Q^i_M)} \sum_{l\in C_j\cap Q^i_M}u^\e_l=: u^{\e,i,j}\qquad \hbox{ if } k\in C_j\cap Q^i_M,
$$
$$
u^{\e,i}_k={1\over \#(A^j_l\cap Q^i_M)} \sum_{l\in A^j_l\cap Q^i_M}u^\e_l =: u^{\e,i,j}_l\qquad \hbox{ if } k\in A^j_l\cap Q^i_M
$$
for $j=1,\ldots, N$ and $l\in I_j$,
and $u^{\e,i}_k= u^\e_k$ if $k\in Q^i_M\setminus \bigcup_{j=1}^N A_j$.

We can now use Lemma \ref{l_pwc} with $u=u^\e$ and $v$ equal to the function defined by $u^{\e,i}$ on $Q^i_M$,
and note that
\begin{eqnarray*}
\sum_{k\in A_j}\e^d|u^\e_k-v_k|^p&=& \sum_i\Bigl(\sum_{k\in C_j\cap Q^i_M}\e^d|u^\e_k-v_k|^p+
\sum_{k\in (A_j\setminus C_j)\cap Q^i_M}\e^d|u^\e_k-v_k|^p\Bigr)\\
&=&
\sum_i\Bigl(\sum_{k\in C_j\cap Q^i_M}\e^d|u^\e_k-u^{\e,i,j}|^p+
\sum_{l\in I_j}\sum_{k\in A^j_l\cap Q^i_M}\e^d|u^\e_k-u^{\e,i,j}_l|^p\Bigr)\\
&\le&
CM^p\e^p\sum_i \Bigl(\sum_{k\in C_j\cap Q^i_M, k-k'\in P_j}\e^d\Bigl|{u^\e_k-u^\e_{k'}\over\e}\Bigr|^p+
\sum_{l\in I_j}\sum_{k\in A^j_l\cap Q^i_M}\e^d\Bigl|{u^\e_k-u^\e_{k'}\over\e}\Bigr|^p\Bigr)
\\
&\le&
CM^p\e^pF_\e(u^\e).
\end{eqnarray*}
We then have
\begin{eqnarray*}
&&\sum_{Q^i_M\subset {1\over\e}\Omega}\e^d\Biggl(\sum_{(k,k')\in N_0(Q^i_M)}\e^p f\Bigl(k,k',{u^\e_k-u^\e_{k'}\over\e}\Bigr)+\sum_{k\in Z(Q^i_M)} g(k, u^\e_k)\Biggr)\\
&=& \sum_{Q^i_M\subset {1\over\e}\Omega}\e^d
\Biggl(\sum_{(k,k')\in N_0(Q^i_M)}f(k,k',{u^\e_k-u^\e_{k'}})+\sum_{k\in Z(Q^i_M)} g(k, u^\e_k)\Biggr)\\
&\ge& \sum_{Q^i_M\subset {1\over\e}\Omega}\e^d \Biggl(
\sum_{Q^i_M\subset {1\over\e}\Omega}\sum_{(k,k')\in N_0(Q^i_M)}f\Bigl(k,k',{u^{\e,i}_k-u^{\e,i}_{k'}}\Bigr)+\sum_{k\in Z(Q^i_M)} g(k, u^{\e,i}_k)\Biggr)+o(1)
\end{eqnarray*}
as $\e\to0$

Since (a translation of ) $u^{\e,i}$ can be used as a test function for $\varphi_M(u^{\e,i,1}, \ldots, u^{\e,i,N})$
we have
\begin{eqnarray*}
&&
\sum_{Q^i_M\subset {1\over\e}\Omega}\Biggl(\sum_{(k,k')\in N_0(Q^i_M)}
f\Bigl(k,k',{u^{\e,i}_k-u^{\e,i}_{k'}}\Bigr)+\sum_{k\in Z(Q^i_M)} g(k, u^{\e,i}_k)\Biggr)
\ge
M^d\varphi_M(u^{\e,i,1}, \ldots, u^{\e,i,N}).
\end{eqnarray*}

We define the piecewise-constant functions $u^{\e,j}_M$ to be equal to $u^{\e,i,1}$ on each $Q^i_M\subset {1\over\e}\Omega$ and to $0$ otherwise. We then obtain
\begin{eqnarray*}
&&\sum_{Q^i_M\subset {1\over\e}\Omega}\e^d\Bigl(\sum_{(k,k')\in N_0(Q^i_M)}\e^p f\Bigl(k,k',{u^\e_k-u^\e_{k'}\over\e}\Bigr)+\sum_{k\in Z(Q^i_M)} g(k, u^\e_k)\Bigr)\\
&&\ge \int_\Omega \varphi_M(u^{\e,1}_M(x), \ldots, u^{\e,N}_M(x))dx +o(1)
\end{eqnarray*}
as $\e\to0$.

Since
$$
u^{\e,i,j}={1\over \#(C_j\cap Q^i_M)} \sum_{l\in C_j\cap Q^i_M}(\widetilde
u^\e_j)_l$$
where $\widetilde u^\e_j$ converges strongly to $u_j$ in $L^1_{\rm loc}(\Omega;\rr^m)$,
so that also $u^{\e,j}_M$ converges strongly to $u_j$ for all $M$. By the Lebesgue Dominated Convergence Theorem we get
\begin{equation}\label{thite}
\lim_{\e\to 0}\int_\Omega \varphi_M(u^{\e,1}_M(x), \ldots, u^{\e,N}_M(x))dx
=
\int_\Omega \varphi_M(u_1(x), \ldots, u_N(x))dx.
\end{equation}

Summing up the liminf inequalities for all $F^j_\e$, (\ref{sete}) and (\ref{thite}), we get
\begin{eqnarray*}
\liminf_{\e\to 0} F_\e(u_\e)&\ge&\sum_{j=1}^N\liminf_{\e\to 0} F^j_\e(u_\e)+m|\Omega|+
\int_\Omega \varphi_M(u_1,\ldots, u_N)\dx\\
&\ge&\sum_{j=1}^N\int_\Omega f^j_{\hom} (\nabla u_j)\dx+m|\Omega|+
\int_\Omega \varphi_M(u_1,\ldots, u_N)\dx,
\end{eqnarray*}
from which (\ref{maineq}) follows taking the limit as $M\to+\infty$ and using Lebesgue's Theorem once again.

\section{Upper bound}
We prove the upper bound for a linear target function
$$
u(x)= (\xi^1 x,\ldots,\xi^N x),
$$
the proof for an affine function following in the same way.
For piecewise-affine functions the same argument applies locally,
while for an arbitrary target function we proceed by approximation (see \cite{BCP}).

A recovery sequence for $u$ can be constructed as follows:

$\bullet$ for all $j=1,\ldots, N$ we choose a recovery sequence $u^j_\e\to \xi^jx$ for $F^j_\e$;
we may regard $u^j_\e$ as defined in the whole $\ZZ^d$.
We set
\begin{equation}\label{jco}
u^\e_k= (u^j_\e)_k\qquad\hbox{ on } C_j;
\end{equation}
\smallskip

$\bullet$ for each fixed $M$ let $Q^i_M$ be the corresponding partition of $\ZZ^d$. For all $i$ we define
$$
u^{\e,i,j}={1\over \#(C_j\cap Q^i_M)} \sum_{l\in C_j\cap Q^i_M}(u^\e_j)_l
$$
for $j=1,\ldots, N$, and take a minimum point $v^{\e,i}$ for  $\wif_M(u^{\e,i,1},\ldots, u^{\e,i,N})$.
We define
$$
u^\e_k=v^{\e, i}_{k-iM}\qquad\hbox{ on } Q^i_M\setminus \bigcup_{j=1}^NA_j ;
$$
Notice that the function $u^{\e,i,j}-u^\e_j$ is of order $\e M$  on $C_j$ ,
and thus, by Lemma \ref{l_pwc}, the difference
\begin{equation}\label{disc}
\sum_{Q^i_M\subset {1\over\e}\Omega}\e^d
\sum_{(k,k')\in N_0(Q^i_M)}f(k,k',{u^\e_k-u^\e_{k'}})-\sum_{Q^i_M\subset {1\over\e}\Omega}\e^d
\sum_{(k,k')\in N_0(Q^i_M)}f(k,k',{\hat u^\e_k-\hat u^\e_{k'}})=o(1)
\end{equation}
as $\e\to0$; here $\hat u^\e_k$ stands for the function equal to  $u^{\e,i,j}$
on $C_j\cap Q^i_M$ and to $u^\e_k$ on ${1\over\e}\Omega\setminus\bigcup_{j=1}^N C_j$.

$\bullet$ for any connected component $A_l^j$ of $A_j\setminus C_j$ with  $A_l^j\subset Q^ i_M$ define
\begin{equation}\label{ueaij}
u_k^\e =v^{\e, i}_{k-iM}+\e z_k^{j,l},
\end{equation}
$z^{j,l}$ being a minimizer of (\ref{mji}).
%$$
%\sum\limits_{\begin{array}{c}\scriptstyle k,k'\in K_l^j\\[-1mm] \scriptstyle k-k'\in P_j\end{array}}
%f(k,k',z_k-z_{k'}).
%$$
Note that $v^{\e, i}_{k-iM}$ is a constant function on $A_l^j$, so that $u_k^\e$ is still minimizing.

With this definition of $u^\e$ we have a recovery sequence for $u$. In order to check that, we introduce an outer approximation of the set $\Omega$ as $\Omega_{\e,M}$ defined by
$$
\Omega_{\e,M}=\bigcup_{i\in I^M_\e}\e Q^i_M,\qquad I^M_\e=\Bigl\{i\in \ZZ^d : Q^i_M\cap {1\over\e}\Omega\neq\emptyset\Bigl\}.
$$
In this way we have
\begin{eqnarray*}
F_\e(u^\e)&\le& F_\e\Bigl(u^\e,{1\over\e}\Omega_{\e,M}\Bigr)\\
&\le&\sum_{j=1}^N F^j_\e\Bigl(u^\e,{1\over\e}\Omega_{\e,M}\Bigr)
+\sum_{j,l: A^j_l\cap {1\over\e}\Omega_{\e,M}\neq\emptyset}\sum\limits_{\begin{array}{c}\scriptstyle k,k'\in A_l^j\\[-1mm] \scriptstyle k-k'\in P_j\end{array}}
\e^df\Bigl(k,k',{u^\e_k-u^\e_{k'}\over\e}\Bigr)
\\&&
+\sum_{i\in  I^M_\e}\e^d\Biggl(\sum\limits_{\begin{array}{c}\scriptstyle k,k'\in Q^i_M\\[-1mm] \scriptstyle k-k'\in P_0\end{array}}
f(k,k',{u^\e_k-u^\e_{k'}})+ \sum_{k\in Q^i_M} g(k, u^\e_k)\Biggr)
\\
&&+\sum_{i\in  I^M_\e}\sum\limits_{\begin{array}{c}\scriptstyle k\in Q^i_M, k'\not\in Q^i_M\\[-1mm] \scriptstyle k-k'\in P_0\end{array}}
\e^d f(k,k',{u^\e_k-u^\e_{k'}}),
\end{eqnarray*}
where we have separated the estimates for the contribution of the infinite components of the hard phases, the isolated islands of hard phases, the contributions of the soft-phase energy and the potential $g$ inside each cube $Q^i_M$ and the contributions of the soft-phase interactions at the boundary of each cube.

We separately examine each term. By (\ref{jco}) and the limsup inequality for $F^j_\e$ we have
\begin{equation}\label{hot}
F^j_\e\Bigl(u^\e,{1\over\e}\Omega_{\e,M}\Bigr)=F^j_\e\Bigl(u^j_\e,{1\over\e}\Omega_{\e,M}\Bigr)\le
F^j_\e\Bigl(u^j_\e,{1\over\e}\Omega'\Bigr)\le F^j_{\hom}(\xi^j x, \Omega')+ o(1)
\end{equation}
for all fixed $\Omega'\supset\!\supset\Omega_{\e,M}$.

As for the second term, we have two cases:

\noindent
$\bullet$ $A^j_l\subset Q^i_M$ for some $i\in I^M_\e$.
In this case by (\ref{ueaij}) we have
\begin{equation}\label{tet}
\sum\limits_{\begin{array}{c}\scriptstyle k,k'\in A_l^j\\[-1mm] \scriptstyle k-k'\in P_j\end{array}}
f\Bigl(k,k',{u^\e_k-u^\e_{k'}\over\e}\Bigr)=
\sum\limits_{\begin{array}{c}\scriptstyle k,k'\in A_l^j\\[-1mm] \scriptstyle k-k'\in P_j\end{array}}
f(k,k',{z^{j,l}_k-z^{j,l}_{k'}})=m^j_l,
\end{equation}
so that
\begin{equation}
\sum_{i\in I^M_\e}\sum_{j,l: A^j_l\subset Q^i_M}\sum\limits_{\begin{array}{c}\scriptstyle k,k'\in A_l^j\\[-1mm] \scriptstyle k-k'\in P_j\end{array}}
\e^df\Bigl(k,k',{u^\e_k-u^\e_{k'}\over\e}\Bigr)\le \sum_{j,l: A^j_l\cap {1\over\e}\Omega_{\e,M}\neq\emptyset}\e^dm^j_l
\le m|\Omega|+o(1);
\end{equation}
\noindent$\bullet$ for the other $A^j_l$ we have $u^\e_k-u^\e_{k'}=0$ for all $k,k'$, so that their total contribution is $O(1/M)$.
 %and then negligible as $M\to+\infty$.

By  (\ref{disc}) the third term is estimated by
\begin{eqnarray}\nonumber&&
\sum_{i\in  I^M_\e}\e^d\Biggl(\sum\limits_{\begin{array}{c}\scriptstyle k,k'\in Q^i_M\\[-1mm] \scriptstyle k-k'\in P_0\end{array}}
f(k,k',{u^\e_k-u^\e_{k'}})+ \sum_{k\in Q^i_M} g(k, u^\e_k)\Biggr)\\
&\le&\nonumber
\sum_{i\in  I^M_\e}\e^d\Biggl(\sum\limits_{\begin{array}{c}\scriptstyle k,k'\in Q^i_M\\[-1mm] \scriptstyle k-k'\in P_0\end{array}}
f(k,k',{\widehat u^\e_k-\widehat u^\e_{k'}})+ \sum_{k\in Q^i_M} g(k, \widehat u^\e_k)\Biggr) +o(1)
\\
&=&\nonumber
\sum_{i\in  I^M_\e}\e^dM^d\wif_M(u^{\e,i,1},\ldots, u^{\e,i,N}) +o(1)
\\
&\le&\label{pol}
\int_{\Omega'}\wif_M(u^{\e,M}_1,\ldots, u^{\e,M}_N)\dx+ o(1),
\end{eqnarray}
where $u^{\e,M}_j$ is the above-defined piecewise-constant function with value $u^{\e,i,j}$ on $Q^i_M$.
Note that
\begin{equation}\label{coed}
u^{\e,M}_j\to \xi^j x\hbox{ in }L^p(\Omega';\rr^m)
\end{equation}
as $\e\to 0$ for all $j$ and $M$.

As for the last term, we note that the difference $u^\e_k-u^\e_{k'}$ is either equal to $0$ (if both $k$ and $k'$ do not belong to in any $C_j$ $j=1,\ldots, N$), to $(u^\e_j)_k$ if $k\in C_j$ and $k'\not\in \bigcup_j C_j$, or to $(u^\e_j)_k-(u^\e_{j'})_{k'}$ if $k\in C_j$ and $k'\in C_{j'}$ with $j\neq j'$. In any case, we can estimate the total contribution by\begin{equation}\label{cof}
C\sum_{i\in I^\e_M} \sum_{j=1}^N \sum_{k\in C_j\cap (iM+(Q_M\setminus Q_{M-R}))}\e^d(1+|(u^\e_j)_k|^p)=
C\sum_{i\in I^\e_M} \sum_{j=1}^N \sum_{k\in C_j\cap (iM+(Q_M\setminus Q_{M-R}))}\e^d(1+|(\widetilde u^\e_j)_k|^p).
\end{equation}
Note that since $\widetilde u^\e_j$ are equi-integrable the latter term vanishes as $M\to+\infty$ uniformly in $\e$. In fact, it can be written as an integral over a set of measure of order $1/M$.

Taking into account this last estimate, together with (\ref{hot}), (\ref{tet}) and (\ref{pol}),
we get
\begin{equation}
\limsup_{\e\to0}F_\e(u_\e)\le
 \sum_{j=1}^N F^j_{\hom}(\xi^j x,\Omega')+ m|\Omega'|+\int_{\Omega'} \wif_M(u)\dx+o(1)
\end{equation}
as $M\to+\infty$.
We can then let $M\to+\infty$ and use Lebesgue's Theorem to obtain
\begin{equation}
\limsup_{\e\to0}F_\e(u_\e)\le
 \sum_{j=1}^N F^j_{\hom}(\xi^j x,\Omega')+ m|\Omega'|+\int_{\Omega'} \varphi(u)\dx.
\end{equation}
Eventually we obtain the desired inequality by the arbitrariness of $\Omega'\supset\!\supset\Omega$.

\section{The dynamical case}
We consider the asymptotic behavior of solutions for the gradient flow with respect to the $L^2$-metric of the functionals
\begin{eqnarray}\nonumber\label{fep0}
F_\e(u)=F_\e\Bigl(u,{1\over\e}\Omega\Bigr)&=&\sum_{j=1}^N\sum_{(k,k')\in N^\e_j(\Omega)} \e^d f\Bigl(k,k', {u_k-u_{k'}\over\e}\Bigr)\\
&&+
\sum_{(k,k')\in N^\e_0(\Omega)} \e^{d+p} f\Bigl(k,k', {u_k-u_{k'}\over\e}\Bigr);
%+\sum_{k\in Z^\e(\Omega)} \e^d g(k, u_k),
\end{eqnarray}
i.e., functionals (\ref{fep}) with $g=0$, with given initial data functions $u^{\e}_0:\ZZ^d\cap{1\over\e}\Omega\to\rr^m$ converging to some $u_0:\Omega\to\rr^m$ (note that in this notation $0\in\NN$ has the meaning of an initial time, and should not be confused with an index $0\in\ZZ^d$ as in the notation labelling the values of discrete functions). To that end, we will apply the minimizing-movement scheme along a sequence of functionals (see \cite{2013LN,AGS}): with fixed $\tau>0$ we define recursively,
for $l\in\NN$, $l\ge 1$, $u^{\e,l}$ as the minimizers of
\begin{equation}
v\mapsto F_\e(v)+{1\over 2\tau}\|v-u^{\e,l-1}\|^2,
\end{equation}
%and space/time-discrete trajectories $u^{\e}(t)= u^{\e,\lfloor t/\tau\rfloor}$.
where $u^{\e,0}=u^\e_0$.
We want to characterize the limits $u^l$ of these minimizers as $\e\to 0$ as the minimizers obtained by recursively applying the same scheme to a $\Gamma$-limit $F_0$; i.e, to show that $u^l$ is a minimizer of
\begin{equation}
v\mapsto F_0(v)+{1\over 2\tau}\|v-u^{l-1}\|^2.
\end{equation}
The norm in these formulas is the $L^2$-norm in $\Omega$.

\medskip
Note that this characterization does not follow trivially from the fundamental theorem of $\Gamma$-convergence since the additional term may not be a continuous perturbation, depending on the topology chosen (e.g, the one used in Theorem \ref{dph}). In order to have a topology for which the last term gives a continuous perturbation, and the sequences $u^{\e,l}$ are still pre-compact, we will use two-scale convergence, also describing the limit behavior of function of the soft phase.

\medskip
Everywhere in this section we assume that the sites that do not interact at all with infinite components of the hard phases do not contribute to the energy functional.
In other words,
\begin{equation}\label{conne}
\hbox{for any } k\in\bigcup\limits_{j=0}^N A_j \ \  \hbox{there exists  } k'\in \bigcup\limits_{j=1}^N C_j\ \  \hbox{such that }\
k \hbox{ and } k' \ \  \hbox{are connected;}
\end{equation}
i.e., either $k=k'$ or there exists a path $\{k_n\}_{n=0,\ldots,K}$
such that $k_0=k$, $k_K=k'$ and $(k_n,k_{n-1})\in \bigcup_{j=0}^N N_j$.

\subsection{$\Gamma$-limits with respect to discrete two-scale convergence}

Let $v^\e:Z^\e(\Omega)\to \rr$ be a sequence bounded in $L^2(\Omega)$.
We say that $v^\e$ {\em weakly} (respectively, {\em strongly}) {\em (discrete) two-scale converges} to the family $\{v^y\}$ for $y\in Y:=\{1,\ldots, T\}^d$ with $v^y\in L^2(\Omega)$ if for all $y\in Y$ the sequence $v^{\e,y}$ of discrete functions obtained by considering only the values $v^\e_k$ with $k=y$ modulo $Y$ weakly (respectively, strongly) converges to the corresponding $v^y$; more precisely, we
define $v^{\e,y}$ on $T\ZZ^d$ as
$$
v^{\e,y}_j= v^\e_{y+j}
$$
for $j\in T\ZZ^d$, and require that its piecewise-constant interpolation weakly converges in $L^2(\Omega)$ to $v^y$.

It can be checked that the definition corresponds to that of  two-scale convergence as in \cite{Ng,Al}; i.e. (for weak convergence) that
for all  families $\{\varphi_k\}_{k\in \ZZ^d}$ of smooth functions $T$-periodic in $k$ we have
\begin{equation}
\lim_{\e\to 0} \sum_{k\in Z^\e(\Omega)}\e^d v^\e_k \varphi_k(\e k)= {1\over T^d}\sum_{k\in Y}\int_\Omega v_k(x)\varphi_k(x)\,dx.\end{equation}
Note that this is equivalent to
\begin{equation}
\lim_{\e\to 0} \int_\Omega v^\e(x) \varphi_k(x)\dx= {1\over T^d}\sum_{y\in Y}\int_\Omega v_y(x)\varphi_y(x)\,dx
\end{equation}
upon identification of $v^\e$ with its piecewise-constant interpolation.

\bigskip

We can compute the $\Gamma$-limit of
$$
G_\e(u)=F_\e(u)+\sum_{k\in Z^\e(\Omega)} \e^d g(u_k-w^\e_k)
$$
with respect to the weak two-scale convergence $u^\e\to \{u^y\}$, where $g:\rr^m\to \rr$ is a continuous function and $w^\e$ strongly  two-scale converges to $\{w^y\}$.

%We suppose that $g^\e$ strongly converges to some $g^y=g^y(x,u)$ periodic in $y$; i.e., that if $u^\e$ converges strongly two-scale to $\{u^y\}$ then
%$g^\e(u^\e)$ converges strongly  two-scale to $\{g^y(\cdot,u^y)\}$. ???

\begin{theorem}
%We suppose that $g^\e(u)$ strongly converges to some $g^y(x,u)$ periodic in $y$; i.e., É..
The $\Gamma$-limit of $G_\e$ with respect to weak discrete two-scale convergence is
\begin{eqnarray}\nonumber
G_0(\{u^y\})&=&\sum_{j=1}^N{1\over \#(C_j\cap Y)}\sum_{y\in C_j\cap Y}\int_\Omega f^j_{\hom}(\nabla u^y)\dx
\\
&&+
{1\over T^d}\sum_{y\in Y\cap\bigcup_{j=1}^N C_j} \int_\Omega g(u^y(x)-w^y(x))\,dx
+\int_\Omega \varphi_g (x,\{u^y(x)\})\dx
\end{eqnarray}
with the constraint that $u^y$ is independent of $y$ on each $C_j$, and $\varphi_g $ is given by
\begin{eqnarray}\nonumber
\varphi_g (x,\{u^y\})&=&\lim_{M\to+\infty} {1\over T^dM^d}
\inf\Bigl\{ \sum_{(k,k')\in N_0(Q_{{T}M})} f(k,k', {v_k-v_{k'}})\\
&&\qquad\qquad+\sum_{k\in Z_0(Q_{{T}M})} g(v_k-w^k(x)):
\sum_{k\in Q_{{T}M}\cap (y+T\ZZ^d)} v_k=M^d u^y
\Bigr\},
\end{eqnarray}
where each test function $v$ is extended by ${{T}M}$-periodicity.

\end{theorem}

\begin{proof}
The proof follows that of Theorem \ref{dph}, with a different characterization of the interaction energy density $\varphi_g$ in terms of the variables $\{u^y\}$. The changes follow the ones for the corresponding theorem in the continuum \cite{BCP} Section 7.2.
\end{proof}

\begin{proposition}\label{perco}
If $f$ and $g$ are convex then
$$
\varphi_g (x,\{u^y\})={1\over T^d}\Bigl(\sum_{(k,k')\in {N^\#_0(Q_{T})}} f(k,k', {u^k-u^{k'}})+\sum_{k\in Z(Q_{T})} g(u^k-w^k(x))\Bigr),
$$
where
$$
N^\#_0(Q_{T})=\{(k,k')\in N_0: \ k\in Q_T\}.
$$
\end{proposition}

\begin{proof}
The proof follows by a classical argument for periodic convex minimization problems (see \cite{BDF} Section 14.3), noting that by Jensen's inequality we may take $M=1$ and a test function $v$ replaced by its mean value on each $y$.  By the average constraint in the definition of $\varphi_g$ this argument fixes exactly the value equal to $u^y$ on each $y$. The definition of $N^\#_0(Q_{T})$ is given so as to avoid double counting in the computation of the interactions.
\end{proof}

\begin{example}\rm In order to illustrate the difference with Theorem \ref{dph} we consider Example \ref{exh}(2).
In that case $C_0\cap Y$ is the only point $1$, so that weak discrete two-scale convergence reduces to the separate weak convergence of even and odd interpolations, and then, by the coerciveness on even interpolations, to the
strong convergence of even interpolations and the weak convergence of odd interpolations. The $\Gamma$-limit is then expressed by
$$
G_0(u^1,u^2)=2\int_{(0,1)}|(u^2)'|^2\dx+ \int_{(0,1)}|u^2-u^1|^2\dx+ {1\over 2} \int_{(0,1)} |u^1-u_0|^2\dx+
{1\over 2} \int_{(0,1)} |u^2-u_0|^2\dx,
$$
where $u^1$ is the limit of odd interpolations and $u^2$ the limit of even interpolations.
Note that the computation of the minimum
$$
\min\Bigl\{{1\over2}|u^1-u_0|^2+ |u^2-u^1|^2: u^1\in \rr\Bigr\}
$$
gives the integrand in the limit of Example \ref{exh}(2).
\end{example}

%\subsubsection{Strong two-scale convergence of discrete functions}

%Let $v^\e:Z^\e(\Omega)\to \rr$ be a sequence bounded in $L^2(\Omega)$.
%We say that $v^\e$ {\em strongly two-scale converges} to the family $\{v_k\}_{k\in \ZZ^d}$ with $v_k\in L^2(\Omega)$ and $T$-periodic in $k$ if
%\begin{equation}
%\lim_{\e\to 0} \int_\Omega |v^\e(x)- v_{\lfloor x/\e\rfloor}(x)|^2\dx=
%\lim_{\e\to 0} \int_\Omega |v^\e_{\lfloor x/\e\rfloor}- v_{\lfloor x/\e\rfloor}(x)|^2\dx=0
%\end{equation}
%where $Y=\{1,\ldots, T\}^d$.

\begin{lemma}\label{stroco}
Let $g^\e_k(u)= C|u-w^\e_k|^2$ with $w^\e$ strongly two-scale converging to $w^y$ and $\sup_\e F_\e(w^\e)<+\infty$.
Then the recovery sequences for $G_0$ converge strongly.
\end{lemma}

\Proof
Take $u^\e$ a recovery sequence for $\{u^y\}$. Note first that since $w^\e$ converges strongly then $|u^\e|^2\dx$ cannot concentrate on the boundary of $\Omega$, otherwise also the $\Gamma$-limit would have a term taking $\partial\Omega$ into account. We then have to show strong convergence in the interior of $\Omega$.

Let $\{Q_\delta\}$ be a family of disjoint cubes of size $\delta$ contained in $\Omega$. We can then write
\begin{eqnarray}\label{stuno}
\nonumber
G_0(\{u^y\})&=& \lim_{\e\to0} \Bigl(F_\e(u^\e)+C\sum_{k\in Z^\e(\Omega)}\e^d |u^\e_k-w^\e_k|^2\Bigr)
\\ \nonumber
&\ge& \sum_{\{Q_\delta\}}  \liminf_{\e\to0}\Bigl(F_\e(u^\e, Q_\delta)+C\sum_{k\in Z^\e(Q_\delta)}\e^d |u^\e_k-w^\e_k|^2\Bigr)
\\ \nonumber
&\ge& \sum_{\{Q_\delta\}}\Bigl(\sum_{j=1}^N{1\over \#(C_j\cap Y)}\sum_{y\in C_j\cap Y}\int_{Q_\delta} f^j_{\hom}(\nabla u^y)\dx
\\&& \nonumber
+
{C\over T^d}\sum_{y\in \bigcup_{j=1}^N C_j} \int_\Omega |u^y(x)-w^y(x)|^2\,dx\Bigr)
\\&&
+ \sum_{\{Q_\delta\}}\liminf_{\e\to0}\e^d
\Bigl( \sum_{(k,k')\in N_0(Q_{\delta/\e})} f(k,k', {u^\e_k-u^\e_{k'}})+\sum_{k\in Z(Q_{\delta/\e})} C|u^\e_k-w_k^\e|^2\Bigr).
\end{eqnarray}
In order to estimate the last term, for all $y\in Y$ and  $k$ with $k-y\in T\ZZ^d$ we substitute $u^\e_k$
with the average $u^{\e,y}$ over all $k'\in Q_{\delta/\e}$ with $k'-y\in T\ZZ^d$. Note that we may suppose that
$\delta/\e\in T\ZZ$, up to a vanishing error in the computation of these averages as $\e\to 0$, so that
$$
u^{\e,y}={T^d\e^d\over \delta^d} \sum_{k'\in Q_{\delta/\e}\cap(y+T\ZZ^d)}u^\e_{k'}\,.
$$
In the following for all $k$ we indicate by $y=y_k$ the (unique) point in $Y\cap (k+T\ZZ^d)$.

With fixed $\eta$,  by using the Young inequality and the convexity inequality on the first term, we then obtain

\begin{eqnarray}
&&\e^d \sum_{\{Q_\delta\}}\Bigl(\sum_{(k,k')\in N_0(Q_{\delta/\e})} f(k,k', {u^\e_k-u^\e_{k'}})+\sum_{k\in Z(Q_{\delta/\e})} C|u^\e_k-w^\e_k|^2\Bigr)
\nonumber
\\
&\ge& \e^d \sum_{\{Q_\delta\}}\Bigl(\sum_{(k,k')\in N_0(Q_{\delta/\e})} f(k,k', {u^\e_k-u^\e_{k'}})+\sum_{k\in Z(Q_{\delta/\e})} C(1-\eta)|u^\e_k-w^{\e,y}|^2
\nonumber \\
&& - C\Bigl({1\over\eta}-1\Bigr)\sum_{k\in Z(Q_{\delta/\e})} |w^\e_k-w^{\e,y}|^2\Bigr)
\nonumber
\\
&\ge& \e^d(1-\eta) \sum_{\{Q_\delta\}}\Bigl(\sum_{(k,k')\in N_0(Q_{\delta/\e})} f(k,k', {u^\e_k-u^\e_{k'}})+\sum_{k\in Z(Q_{\delta/\e})} C|u^\e_k-w^{\e,y}|^2
\nonumber \\
&&
-\sum_{k\in Z(Q_{\delta/\e})} C|u^{\e,y}-w^{\e,y}|^2+\sum_{k\in Z(Q_{\delta/\e})} C|u^{\e,y}-w^{\e,y}|^2- {1\over\eta}\sum_{k\in Z(Q_{\delta/\e})} C|w^\e_k-w^{\e,y}|^2\Bigr)
\nonumber
\\
&=& \e^d(1-\eta) \sum_{\{Q_\delta\}}\Bigl(\sum_{(k,k')\in N_0(Q_{\delta/\e})} f(k,k', {u^\e_k-u^\e_{k'}})
+\sum_{k\in Z(Q_{\delta/\e})} C\bigl(|u^\e_k|^2-|u^{\e,y}|^2\bigr)
\label{oups}
\\
&&+{\delta^d\over\e^dT^d}\sum_{y\in Y} C|u^{\e,y}-w^{\e,y}|^2- {1\over\eta}\sum_{k\in Z(Q_{\delta/\e})} C|w^\e_k-w^{\e,y}|^2\Bigr)
\nonumber \\
\nonumber
&\ge & (1-\eta){\delta^d\over T^d}\sum_{\{Q_\delta\}}\Bigl(
\sum_{(y,y')\in N^\#_0(Y)} f(y,y', {u^{\e,y}-u^{\e,y'}})+\sum_{y\in Y} C|u^{\e,y}-w^{\e,y}|^2\Bigr)
-C'\frac{\e}{\delta}\delta^{d} \\
\nonumber
&&+\e^d(1-\eta)\sum_{\{Q_\delta\}}\sum_{k\in Z(Q_{\delta/\e})} C(|u^\e_k|^2-|u^{\e,y}|^2)- {(1-\eta)\over\eta}\e^d\sum_{\{Q_\delta\}}\sum_{k\in Z(Q_{\delta/\e})} C|w^\e_k-w^{\e,y}|^2.
\end{eqnarray}

Note that by taking into account only interactions with
$(k,k')\in N_0(Q_{\delta/\e})$ we have neglected some interactions ``through the boundary'' of $Q_{\delta/\e}$, which introduce an error on the boundary of the hard components. After a proper adjustment of the position of $Q_\delta$ this can be estimated by the convexity and the Poincar\'e inequality, which gives the term $-C'\e\delta^{d-1}$.
Indeed, by (\ref{c2}) and the Poincar\'e inequalities on the first hard phase $C_1$ we have
$$
\e^d \sum_{\{Q_\delta\}}\sum\limits_{k\in C_1\cap Z(Q_{\delta/\e}) }|u^\e_k-C^\e|^p\leq C
$$
for some constant $C^\e$, we can take $C^\e$  equal to the average of  $u^\e_k$ over
$\bigcup_{\{Q_\delta\}}( C_1\cap Z(Q_{\delta/\e}) )$.  Denote by $\hat\mathcal{C}_1$ the set of $k\in\mathbb Z^d$
that are connected to $C_1$.  Combining the last estimate with the energy bound and considering (\ref{c2}) we get
$$
\e^d \sum_{\{Q_\delta\}}\sum\limits_{k\in  \hat\mathcal{C}_1\cap  Z(Q_{\delta/\e}) }|u^\e_k-C^\e|^p\leq C
$$
%$$
%\e^d \sum_{\{Q_\delta\}}\sum\limits_{k\in(\cup C_j^\e)\cap Q_\delta }|u^\e_k|^p\leq C.
%$$

Next, we choose $\mathcal{R}$ such that  any two points do not interact if the distance between them is greater than or equal to $\mathcal{R}$.
For each $\e>0$ one can adjust the position of the cubes $Q_{\delta/\e}$ in such a way that
$$
\e^d  \sum_{\{Q_\delta\}} \sum\limits_{k\in \hat\mathcal{C}_1\cap  Z(Q^{\mathcal{R}}_{\delta/\e})}|u^\e_k-C^\e|^p\leq C\big._\mathcal{R}\frac{\e}{\delta},
$$
where $$Z(Q^{\mathcal{R}}_{\delta/\e})=\{k\in Z(Q_{\delta/\e})\,:\, \mathrm{dist}(k,\partial Q_{\delta/\e})\leq \mathcal{R}\}.$$
Setting $\widehat{\mathcal{N}}_0(Q_{\delta/\e})=N_0(Q_{\delta/\e})\cap(\mathcal{C}_1\times\mathcal{C}_1)$, with the help of Jensen's inequality we obtain
$$
\e^d \sum_{\{Q_\delta\}}\sum_{(k,k')\in \widehat{\mathcal{N}}_0(Q_{\delta/\e})} f(k,k', {u^\e_k-u^\e_{k'}})=
\e^d \sum_{\{Q_\delta\}}\sum_{(k,k')\in \widehat{\mathcal{N}}_0(Q_{\delta/\e})} f(k,k', (u^\e_k-C^\e)-(u^\e_{k'}-C^\e))
$$
$$
\geq  {\delta^d\over T^d}\sum_{\{Q_\delta\}}
\sum_{(y,y')\in (N^\#_0(Y)\cap(\mathcal{C}_1\times\mathcal{C}_1)} f(y,y', {(u^{\e,y}-C^\e)-(u^{\e,y'}-C^\e)})
-C'\frac{\e}{\delta}\delta^{d}
$$
 $$
=  {\delta^d\over T^d}\sum_{\{Q_\delta\}}
\sum_{(y,y')\in (N^\#_0(Y)\cap(\mathcal{C}_1\times\mathcal{C}_1)} f(y,y', {u^{\e,y}-u^{\e,y'}})
-C'\frac{\e}{\delta}\delta^{d}.
$$
Considering  (\ref{conne})  and summing up over all the connected components yields
$$
\e^d \sum_{\{Q_\delta\}}\sum_{(k,k')\in N_0(Q_{\delta/\e})} f(k,k', {u^\e_k-u^\e_{k'}})
\geq  {\delta^d\over T^d}\sum_{\{Q_\delta\}}
\sum_{(y,y')\in (N^\#_0(Y)} f(y,y', {u^{\e,y}-u^{\e,y'}})
-C'\frac{\e}{\delta}\delta^{d}.
$$

%
%$$
%\e^d  \sum_{\{Q_\delta\}} \sum\limits_{k\in Q^{\mathcal{R}}_\delta\cap\bigcup_{j=0}^N A_j^\e }|u^\e_k-C^\e|^p\leq %C\big._\mathcal{R}\frac{\e}{\delta};
%$$
%Combining this estimate with  the energy bound we obtain
%$$
%\e^d  \sum_{\{Q_\delta\}} \sum\limits_{k\in Q_\delta\cap\bigcup_{j=0}^N A_j^\e }|u^\e_k|^p\leq C.
%$$
%Assume first that
% $$
% \bigcup\limits_{\{Q_\delta\}} \big(\hat\mathcal{C}_1\cap  Z(Q_{\delta/\e})\big) =\bigcup_{j=0}^N A_j^\e
%$$
 % The deisred inequality easily follows from this estimate.

Passing now in (\ref{oups}) to the limit as $\e\to0$, we obtain the estimate
\begin{eqnarray*}
&&\liminf_{\e\to0}\e^d \sum_{\{Q_\delta\}}
\Bigl( \sum_{(k,k')\in N_0(Q_{\delta/\e})} f(k,k', {u^\e_k-u^\e_{k'}})+\sum_{k\in Z(Q_{\delta/\e})} C|u^\e_k-w^k_\e|^2\Bigr)\\
&\ge & (1-\eta){\delta^d\over T^d}\sum_{\{Q_\delta\}}\Bigl(
\sum_{(y,y')\in N_0(Y)} f(y,y', {u^{y}_\delta-u^{y'}_\delta})+\sum_{k\in Y} C|u^{y}_\delta-w^{y}_\delta|^2\Bigr)
 %-C'\delta^{d+2}
 \\
&&+ C(1-\eta) \liminf_{\e\to0}\e^d\sum_{\{Q_\delta\}}\Bigl(\sum_{k\in Z(Q_{\delta/\e})} (|u^\e_k|^2-|u^{y}_\delta|^2)-{1\over \eta}\sum_{k\in Z(Q_{\delta/\e})} |w^\e_k-w^{\e,y}|^2\Bigr),
\end{eqnarray*}
where the subscript $\delta$ indicates the average on $Q_\delta$.

Note that, using Proposition \ref{perco},
\begin{eqnarray}\label{stdue}
&&\nonumber
\int_{Q_\delta}\varphi_g(x,\{u^y_\delta\})\dx\\
&=& {1\over T^d}\nonumber
\int_{Q_\delta}\Bigl(\sum_{(y,y')\in N^\#_0(Y)} f(y,y', {u^{y}_\delta-u^{y'}_\delta})+\sum_{y\in Y} C|u^{y}_\delta-w^{y}(x)|^2\Bigr)\dx\\
&=& {\delta^d\over T^d}\Bigl(\sum_{(y,y')\in N^\#_0(Y)} f(y,y', {u^{y}_\delta-u^{y'}_\delta})+\sum_{y\in Y} C|u^{y}_\delta-w^{y}_\delta|^2\Bigr)+ O\Bigl(\int_{Q_\delta}|w^{y}_\delta-w^{y}(x)|^2\dx\Bigr).
\end{eqnarray}
Comparing (\ref{stuno}) and (\ref{stdue}), by the arbitrariness of the partition $\{Q_\delta\}$ and $\eta>0$, and noting that $\sum\chi_{Q_\delta}\{u^y_\delta\}$ converge to $\{u^y\}$ as $\delta\to 0$
we then get
$$
\liminf_{\e\to 0} \int_\Omega (|u^{\e,y}|^2- |u^y|^2)\le 0,
$$
which implies the strong convergence for all $y\in Y$.
\qed

\subsection{Minimizing movements}
We now fix initial data $u^\e_0$ strongly converging to $\{u_0^y\}$ and with $\sup_\e F_\e(u^{0}_\e)<+\infty$.
Given $\tau>0$ we define iteratively the functions $u^\e_{\tau,n}$ as the unique minimizers of the
problems
$$
\min\Bigl\{ F_\e(v)+ {1\over 2\tau} \sum_{k\in Z^\e(\Omega)} \e^d|v_k- (u^\e_{\tau,n-1})_k|^2\Bigr\},
$$
where we have set $u^\e_{\tau,0}=u_{0}^\e$.

\begin{theorem} Suppose that $f$ and all $f^j_{\hom}$ are continuously twice differentiable.
For all choices of infinitesimal sequences $\e$ and $\tau$, the functions $u^{\tau,\e}(x,t)$ defined by
$$
u^{\tau,\e}(x,t)= (u^\e_{\tau,\lfloor t/\tau\rfloor})_{\lfloor x/\e\rfloor}
$$
converge in $C^{1/2}((0,+\infty); L^2(\Omega))^{T^d}$ to a vector function $\{u^y\}$ with $y\in Y$.
The components of this function are independent of $y$ on each $C_j\cap Y$, so that we equivalently use the notation $u_j$ for their common value. With this notation and setting
$$
c_j= {\#(C_j\cap Y)\over T^d}
$$
for all $j=0,\ldots, N$, $\{u^y\}$ is characterized as the solution of
the coupled system
\begin{eqnarray}\label{systo}\nonumber
c_j{\partial u_j\over\partial  t} &=&{\rm div} \Bigl(\nabla f^j_{\hom}(\nabla u_j)\Bigr)-\sum_{(y,y')\in N_0(Y), y\in C_j}{\partial \over \partial u} f(y,y', {u_j-u^{y'}})\\
&&+\sum_{(y',y)\in N_0(Y), y\in C_j}{\partial \over \partial u} f(y',y, {u^{y'}-u_j}), \qquad j=1,\ldots, N,
\end{eqnarray}
$$
c_0{\partial u^y\over\partial  t}=- \sum_{(y,y')\in N_0(Y)}{\partial \over \partial u} f(y,y', {u^y-u^{y'}})+\sum_{(y',y)\in N_0(Y)}{\partial \over \partial u} f(y',y, u^{y'}-u^y), \qquad y\in Y\cap C_0
$$
with $u_j$ satisfying Neumann boundary conditions   $$\nabla f^j_{\hom}(\nabla u_j)\cdot\nu=0$$ on $\partial\Omega\times(0,+\infty)$,  and  $u^y$ the coupling condition
\begin{equation}\label{combo2}
u^y = u_j\hbox{ if } y\in C_j\cap Y
\end{equation}
and the initial conditions
$$
u^y(0,x)= u^y_0(x).
$$
This limit function also coincides with the limit of gradient flows of $F_\e$.
\end{theorem}

\begin{proof} By the convexity of the functionals we can use the stability for minimizing movements along $F_\e$.
The results will follow by applying Theorem 11.2 in \cite{2013LN}, provided that we have strong convergence of minimizing sequences (see \cite{2013LN} Remark 11.2).
This follows from Lemma \ref{stroco} applied iteratively with
$$
g^\e_k(u)= {1\over 2\tau} |u-(u^\e_{\tau,n-1})_k|^2,
$$
so that all sequences $u^\e_{\tau,n}$ are strongly converging as $\e\to0$
(thanks to the strong convexity of the $\Gamma$-limit).
If we denote by $\{u_{\tau,n}^y\}$ their two-scale limit, by the fundamental theorem of $\Gamma$-convergence
they solve iteratively an analogous minimization scheme with $u_\e^{\tau,0}=\{u_0^y\}$,
and $\{u_{\tau,n}^y\}$ being the unique minimizer of
\begin{eqnarray}\label{mili}\nonumber
&&\min\Bigl\{ \sum_{j=1}^N{1\over \#(C_j\cap Y)}\sum_{y\in C_j\cap Y}\int_\Omega f^j_{\hom}(\nabla v^y)\dx\\
&&+
{1\over T^d}\sum_{y\in Y} {1\over 2\tau}\int_\Omega |v^y(x)-u_{\tau,n-1}^y(x)|^2\,dx
+\sum_{(y,y')\in N_0(Y)}\int_\Omega  f(y,y', {v^y-v^{y'}})\dx\Bigr\},
\end{eqnarray}
with the constrain that $v^y$ is constant on each component $C_j$.

Under the assumption that $f$ and $f^j_{\hom}$ are $C^2$ we can derive the Euler-Lagrange equations for $\{u_{\tau,n}^y\}$.
It is convenient to separate the hard and soft phases by introducing the functions
\begin{equation}\label{combo}
u^{\tau,n}_j= u_{\tau,n}^y \hbox{ if } y\in C_j\cap Y
\end{equation}
for $j=1,\ldots, N$, and the set of indices $C_0=Y\setminus \bigcup_{j=1}^NC_j$.

For $j=1,\ldots, N$ we obtain
\begin{eqnarray*}
&&-{\rm div} \nabla f^j_{\hom}(\nabla u^{\tau,n}_j)+c_j{u^{\tau,n}_j-u^{\tau,n-1}_j\over \tau}\\
&&+\sum_{(y,y')\in N_0(Y), y\in C_j}{\partial \over \partial u} f(y,y', {u^{\tau,n}_j-u_{\tau,n}^{y'}})-\sum_{(y',y)\in N_0(Y), y\in C_j}{\partial \over \partial u} f(y',y, {u_{\tau,n}^{y'}-u^{\tau,n}_j})=0
\end{eqnarray*}
with Neumann boundary condition, that reads
$$
 \nabla f^j_{\hom}(\nabla u^{\tau,n}_j)\cdot \nu=0 \qquad\hbox{on }\partial \Omega, \quad j=1,\ldots,\,N,
$$
where $\nu$ stands for the exterior normal on $\partial \Omega$.

For fixed $y\in C_0$ we obtain instead
$$
c_0{u^{\tau,n}_y-u^{\tau,n-1}_y\over \tau}+ \sum_{(y,y')\in N_0(Y)}{\partial \over \partial u} f(y,y', {u_{\tau,n}^y-u_{\tau,n}^{y'}})-\sum_{(y',y)\in N_0(Y)}{\partial \over \partial u} f(y',y, u_{\tau,n}^{y'}-u_{\tau,n}^y)=0.
$$
Note the coupling condition (\ref{combo}).

We define the piecewise-constant trajectories
$$
u^{\tau}_j(t,x))= u^{\tau,\lfloor t\tau\rfloor}_j(x)
$$
for $j=1,\ldots, N$, and
$$
u_{\tau}^y(t,x)= u_{\tau,\lfloor t\tau\rfloor}^y(x)
$$
for $y\in C_0$, which converge uniformly in $[0,+\infty)$ as $\tau\to 0$ to functions $u_j(t,x)$ and $u^y(t,x)$,
respectively.
By passing to the limit in the Euler-Lagrange equations we obtain system (\ref{systo}).
\end{proof}

\begin{remark}\rm The limit system is not decoupled also if $C_0=\emptyset$, in which case we have the system of partial differential equations for $u_j$ only
\begin{eqnarray*}
&& c_j{\partial u_j\over\partial  t} ={\rm div} \Bigl(\nabla f^j_{\hom}(\nabla u_j)\Bigr)\\
&&\qquad -\sum_{j'\neq j}\Bigl(\sum_{(y,y')\in N_0(Y), y\in C_j, y'\in C_{j'}}{\partial \over \partial u} f(y,y', {u_j-u_{j'}})\\
&&\qquad +\sum_{(y',y)\in N_0(Y), y\in C_j, y'\in C_{j'}}{\partial \over \partial u} f(y',y, {u_{j'}-u_j})\Bigr).
\end{eqnarray*}
\end{remark}

%\begin{remark}[equivalence with an integro-differential system]\rm
%{\bf what shall we put here?}
%\end{remark}

\begin{example}\rm
In the case of the energies in Example \ref{exh}(2) the limit $(u_1(t,x), u_2(t,x))$ satisfies
$$
\cases{\displaystyle {\partial u_2\over\partial t}= 2{\partial^2 u_2\over\partial x^2}-u_2+u_1\cr\cr
 \displaystyle{\partial u_1\over\partial t}= u_1-u_2\cr\cr
\displaystyle u_1(x,0)= u_2(x,0)=u^0(x)}.
$$
Note that we may solve the ODE and obtain the integro-differential problem satisfied by $u=u_2$ only
$$
\cases{\displaystyle {\partial u(x,t)\over\partial t}=
2{\partial^2 u(x,t)\over\partial x^2}-u(x,t)+u^0(x)e^{-t}+\int_0^t e^{s-t}u(x,s)\,ds\cr\cr
\displaystyle u(x,0)=u^0(x)}
$$
\end{example}

\section{Appendix}

\begin{lemma}\label{l_pwc}
Let $u_k=v_k$ if $k\not\in \bigcup_{j=1}^N A_j$. Then we have
\begin{eqnarray}\label{pwc}\nonumber &&
\sum_{(k,k')\in N_0(\Omega)}\e^d|f(k,k',{u_k-u_{k'}})- f(k,k',{v_k-v_{k'}})|
\\&\le&
C\sum_{j=1}^N
\Bigl(\sum_{k\in A_j}\e^d|u_k-v_k|^p\Bigr)^{1/p} \Bigl(F_\e(u)+F_\e(v)\Bigr)^{p-1/p}
\end{eqnarray}
\end{lemma}

\Proof
We estimate
\begin{eqnarray*}
&&\sum_{(k,k')\in N_0(\Omega)}\e^d|f(k,k',{u_k-u_{k'}})- f(k,k',{v_k-v_{k'}})|\\
&\le& C\sum_{(k,k')\in N_0(\Omega)}\e^d|({u_k-u_{k'}})-({v_k-v_{k'}})| (|u_k-u_{k'}|^{p-1}+ |{v_k-v_{k'}}|^{p-1})
\\
&\le& C\sum_{j=1}^N
\sum_{(k,k')\in N_0(\Omega),k\in A_j}\e^d|u_k-v_k| (|u_k-u_{k'}|^{p-1}+ |{v_k-v_{k'}}|^{p-1})\\
&\le& C\sum_{j=1}^N
\Bigl(\sum_{(k,k')\in N_0(\Omega),k\in A_j}\e^d|u_k-v_k|^p\Bigr)^{1/p}
\Bigl(\sum_{(k,k')\in N_0(\Omega),k\in A_j}(|u_k-u_{k'}|^p+ |{v_k-v_{k'}}|^p)\Bigr)^{(p-1/p)}
\\
&\le& C\sum_{j=1}^N
\Bigl(\sum_{k\in A_j}\e^d|u_k-v_k|^p\Bigr)^{1/p}
\Bigl(\sum_{(k,k')\in N_0(\Omega)}\e^d(|u_k-u_{k'}|^p+ |{v_k-v_{k'}}|^p)\Bigr)^{(p-1/p)}.\\
\end{eqnarray*}
The thesis then follows by (\ref{c2}) and  (\ref{c3}).\qed

\end{document}